\documentclass[reqno,12pt]{amsart}

\usepackage[utf8]{inputenc}

\usepackage{enumerate}
\usepackage[margin=1in]{geometry}
\usepackage{ifpdf}
\usepackage{amsmath}
\usepackage{amsfonts}
\usepackage{amssymb}
\usepackage{amsthm}
\usepackage[ocgcolorlinks,hyperfootnotes=false,colorlinks=true,citecolor=blue,linkcolor=blue,urlcolor=blue]{hyperref}
\usepackage{setspace}
\usepackage{amsrefs}
\usepackage{nicefrac}
\usepackage{graphicx}
\usepackage{color}
\usepackage{mathtools}

\usepackage[T1]{fontenc}

\newcommand{\ignore}[1]{}

\renewcommand{\Re}{\operatorname{Re}}

\newcommand{\Aut}{\operatorname{Aut}}

\newcommand{\abs}[1]{\left\lvert {#1} \right\rvert}
\newcommand{\sabs}[1]{\lvert {#1} \rvert}
\newcommand{\norm}[1]{\left\lVert {#1} \right\rVert}
\newcommand{\snorm}[1]{\lVert {#1} \rVert}

\newcommand{\C}{{\mathbb{C}}}
\newcommand{\R}{{\mathbb{R}}}

\newcommand{\N}{{\mathbb{N}}}

\newcommand{\D}{{\mathbb{D}}}

\newcommand{\bB}{{\mathbb{B}}}

\newcommand{\bH}{{\mathbb{H}}}

\newcommand{\sZ}{{\mathcal{Z}}}

\newtheorem{thm}{Theorem}[section]

\newtheorem{prop}[thm]{Proposition}

\newtheorem{lemma}[thm]{Lemma}

\theoremstyle{definition}

\theoremstyle{remark}

\author{Ji\v{r}\'{\i} Lebl}
\thanks{The author was in part supported by Simons Foundation collaboration grant 710294.}
\address{Department of Mathematics, Oklahoma State University,
Stillwater, OK 74078, USA}
\email{lebl@okstate.edu}

\date{November 20, 2023}

\title{Exhaustion functions and normal forms for proper maps of balls}


\begin{document}


\begin{abstract}
We study a relationship between rational proper maps of balls in different dimensions and strongly plurisubharmonic exhaustion functions of the unit ball induced by such maps.  Putting the unique critical point of this exhaustion function at the origin leads to a normal form for rational proper maps of balls.  The normal form of the map, which is up to composition with unitaries, takes the origin to the origin, and it normalizes the denominator by eliminating the linear terms and diagonalizing the quadratic part.  The singular values of the quadratic part of the denominator are spherical invariants of the map.  When these singular values are positive and distinct, the normal form is determined up to a finite subgroup of the unitary group.  We also study which denominators arise for cubic maps, and when we do not require taking the origin to the origin, which maps are equivalent to polynomials.
\end{abstract}

\maketitle



\section{Introduction} \label{section:intro}

Let $\bB_n \subset \C^n$ be the unit ball.
It is a well-known theorem of Alexander~\cite{Alexander}
that if $f \colon \bB_n \to \bB_n$ is a proper holomorphic map
and $n \geq 2$, then $f$ is an automorphism of $\bB_n$ and hence
a linear fractional map.
If $n=1$, the classical theorem of Fatou says that $f$ equals a finite
Blaschke product.
A natural question is to classify proper holomorphic maps
$f \colon \bB_n \to \bB_N$ when $N \not= n$.  If $N < n$,
no such proper maps exist.  By a result
of Dor~\cite{Dor}, there exist holomorphic proper maps
$f \colon \bB_n \to \bB_{n+1}$, which are continuous up to the
boundary, but not smooth there.
On the other hand, via a theorem of Forstneri\v{c}~\cite{Forstneric}, if
a proper holomorphic map $f \colon \bB_n \to \bB_N$, $n \geq 2$,
extends smoothly up to the boundary, then $f$ is
rational and of degree bounded by an expression of $n$ and $N$.
For a recent comprehensive treatment of rational maps of spheres, see
D'Angelo~\cite{DAngelo:spheresbook}.

Two maps $f$ and $F$ are \emph{spherically equivalent} if 
$F=\tau \circ f \circ \psi$, where $\psi \in \Aut(\bB_n)$ and $\tau \in
\Aut(\bB_N)$.
A quantity invariant under this equivalence is
a \emph{spherical invariant}.
The study of normal forms for rational proper maps for a fixed $n$ and $N$
(under spherical equivalence)
is an area of active research and the normal form is known when the
codimension $N-n$ is small.
For example, Faran~\cite{Faran:B2B3} gives 4 equivalence classes when $n=2$
and $N=3$, up to automorphisms, so-called spherical equivalence.
Faran~\cite{Faran:firstgap} and Huang~\cite{Huang:firstgap} show there is
only one equivalence class when $N < 2n-1$.
Many more cases have been handled,
see for example \cites{Hamada05,HJX06,Ebenfelt13,HJY14} and the references within.
In general, there are many spherical equivalence classes.
When $N \geq 2n$, there are uncountably many:
Indeed, D'Angelo and the author showed in \cite{DL:homotopies},
that while there are only finitely many homotopy equivalence classes,
any homotopy of two spherically inequivalent maps contains
uncountably many spherical equivalence classes.  See also
D'Angelo~\cite{DAngelo:book}.
In this paper we offer a computable normal form for spherical equivalence
up to a subgroup of the unitary group.  Generically, our normal form is up to a
small finite subgroup, but not always.

The approach we take in this paper is to not fix $N$, but to fix
the degree instead.
Suppose that $f$ is rational,
written as $\frac{p}{g}$ where $p \colon \C^n \to \C^N$ and
$g \colon \C^n \to \C$ are polynomials.  If $\frac{p}{g}$ are written
in lowest terms, we consider the degree $d$ of $f$ to be the maximum of the
degrees of $p$ and $g$.  The degree of $g$ is always less than or equal to that of $p$,
and if $p(0) = 0$, then the degree of $g$ is $d-1$ or less,
see D'Angelo~\cite{DAngelo:book}.
The smallest dimension $N'$ such that $f$ maps into an affine $N'$
dimensional subspace is called the \emph{embedding dimension} of $f$.
If we fix the degree $d$, then there is
a maximum embedding dimension for any degree $d$ proper rational map,
it is the number of nonconstant linearly independent monomials of degree
$d$ or less.
We will focus on the denominator.
Given any polynomial that does not vanish on the closed unit ball,
there exists a rational proper map of balls in lowest terms
with the given polynomial as denominator, see
D'Angelo~\cite{DAngelo:spheresbook} (see also
Catlin--D'Angelo~\cite{CatlinDAngelo}).

When the degree $d=1$, every
rational proper map is equivalent to a linear embedding, $z \oplus 0$.
When $d=2$, every rational proper map is equivalent to a polynomial map taking the origin to
the origin by the author's prior work~\cite{Lebl:normal},
and this result also
follows from the theorem below.
So in degrees 1 and 2, the denominator $g$ can be made constant.
In higher degrees, we get the following result.

\begin{thm} \label{thm:normdenom}
Suppose $f \colon \bB_n \to \bB_N$ is a
rational proper map of degree $d$.
Then there exist numbers
$0 \leq \sigma_1 \leq \sigma_2 \leq \cdots \leq \sigma_n \leq \frac{d-1}{2}$,
and automorphisms
$\psi \in \Aut(\bB_n)$ and
$\tau \in \Aut(\bB_N)$
such that
$\tau \circ f \circ \psi = \frac{P}{G}$ (in lowest terms), where
$P(0) = 0$,
$G$ is of degree at most $d-1$, and the homogeneous expansion of $G$ is
\begin{equation}
G(z) = 1 + G_2(z) + G_3(z) + \cdots + G_{d-1}(z), 
\qquad
\text{where}
\quad 
G_2(z) =
\sum_{k=1}^n \sigma_k z_k^2 .
\end{equation}
That is, $G$ has no linear terms, and the quadratic part is diagonalized.
The $\sigma_1,\ldots,\sigma_n$ are
spherical invariants and $f$ is in normal form up to composition
with unitary maps.
\end{thm}

If $F$ and $\Phi$ are spherically equivalent and in the form above,
then the $\sigma_1,\ldots,\sigma_n$ for both maps are equal and
$\Phi = U \circ F \circ V$, where $U$ and $V$ are unitary matrices
and $G_2 \circ V = G_2$.

The normal form is found by locating the critical point of what we call
a $\Lambda$-function.
Let $f = \frac{p}{g}$ be a rational proper
map of balls written in lowest terms.  Define $\Lambda \colon \bB_n \to \R$,
\begin{equation} \label{eq:Lambda}
\Lambda(z,\bar{z}) = \Lambda_f(z,\bar{z}) =
\frac{\sabs{g(z)}^2-\snorm{p(z)}^2}{(1-\snorm{z}^2)^d} .
\end{equation}
Note that D'Angelo--Xiao~\cite{DX} (Theorem 3.2) have used a similar expression to
study groups associated to rational proper maps of balls.
See also Theorem 5.5 in D'Angelo~\cite{DAngelo:herm}.
In this paper, we prove that this $\Lambda$-function relates
rational proper maps of balls to the strong (pseudo)convexity
of the source ball.

\begin{thm} \label{thm:Lambda}
\begin{samepage}
Suppose $f = \frac{p}{g} \colon \bB_n \to \bB_N$ is a rational proper map
of degree $d > 1$ written in lowest terms and define $\Lambda = \Lambda_f$ as in
\eqref{eq:Lambda}.  Then
\begin{enumerate}[(i)]
\item
For $\tau \in \Aut(\bB_N)$, the $\Lambda$-function for $f$ and $\tau \circ f$ are
identical,
$\Lambda_f = \Lambda_{\tau \circ f}$.
\item
If $\psi \in \Aut(\bB_n)$, then
$\Lambda_f \circ \psi = C \Lambda_{f \circ \psi}$ for
a constant~$C$.
\item
$\Lambda$ is a strongly plurisubharmonic exhaustion function for $\bB_n$:
$\Lambda$ is strongly plurisubharmonic and $\Lambda(z)$ goes to $+\infty$ as $z \to \partial \bB_n$.
In fact, $\Lambda$ is strongly convex near $\partial \bB_n$.
\item
$\Lambda$ has a unique critical point (a minimum) in $\bB_n$.
\end{enumerate}
\end{samepage}
\end{thm}

Strong (rather than strict) plurisubharmonicity (resp.\ convexity)
means that not only is the 
complex (resp.\ real) Hessian of $\Lambda$ positive definite on $\bB_n$, but its least
eigenvalue is uniformly bounded from below on $\bB_n$.

The fourth item is the key to proving Theorem~\ref{thm:normdenom}.  
The normal form is attained by
an automorphism $\psi$ that takes the unique critical point of $\Lambda$
to the origin.  Via the uniqueness, the map is normalized up to
the unitary group.  The new denominator will have no linear terms,
and the $\sigma_k$s are 
found by diagonalizing the quadratic part of the denominator via unitaries.
The main step in the normalization is solving $\nabla \Lambda = 0$,
and hence the normalization is computable provided we can solve for
roots of the polynomials involved.

If $0 < \sigma_1 < \sigma_2 < \cdots < \sigma_n$, then
the only possible $V$ such that $G_2 \circ V = G_2$ is the diagonal
matrix with $1$s and $-1$s on the diagonal.
That is, any $z_j$ can be replaced by $-z_j$, and the
set of such $V$ forms
a finite subgroup of $U(n)$ of order $2^n$.
We can also put $P$ into a normal form.
Fix an ordering on the $n$-variable multiindices $\alpha$
and write $P(z) = \sum_{\alpha} c_{\alpha} z^\alpha$, where $c_{\alpha} \in \C^N$.
Choose the unitary matrix $U$ so that the matrix $[c_\alpha]$,
whose columns are $c_\alpha$ ordered as given, is in row echelon form with
positive pivots.
If $N$ is the embedding dimension for $f$, then $[c_\alpha]$ has full
rank and $U$ is unique.
That is, when the $\sigma_k$s are distinct and nonzero,
we have a normal form up to a small finite subgroup of the unitary
group.

When $\sigma_1 = \cdots = \sigma_n=0$ and $G \equiv 1$, the map is
polynomial and takes $0$ to $0$.
D'Angelo proved \cite{DAngelo:book} that if two polynomial maps that fix
the origin are equivalent, then they are equivalent up to unitaries.
Therefore, Theorem~\ref{thm:normdenom} extends D'Angelo's observation to rational maps.
There do exist degree 3 maps not spherically equivalent to a polynomial,
for example, a product of three Blaschke factors is in general not equivalent to
a polynomial proper map of discs, that is, to $z^3$.
For an example from $\bB_2$ to $\bB_4$ see
also Faran--Huang--Ji--Zhang~\cite{FHJZ}.
In section~\ref{section:poly}, we prove that if the 
embedding dimension of the map is maximal for its degree, then there exists
a linear fractional $\tau$ such that $\tau \circ f$ is a polynomial
proper map to the ball $\bB_N$, the generalized ball $\bB_{1,N-1}$, or
the Heisenberg realization of the ball $\bH_N$.  So $f$ is equivalent
to a polynomial if we take possibly a different representation of the
ball and do not assume that $f$ takes the origin to the origin.

When considering the normal form,
it is better to dispense with the target automorphism
group via Lemma~\ref{lemma:targetautform} and consider the normal form for
the polynomials $r(z,\bar{z}) = \sabs{g(z)}^2-\snorm{f(z)}^2$.  For instance,
when $0 < \sigma_1 < \cdots < \sigma_n$, if we put the map, and
therefore $r$, into the form from the theorem, then $r$ is normalized up to
swapping the signs of the $z_k$s.

When $d < 3$, then $G$ must be of degree 1 or less, that is, $\sigma_1=\cdots=\sigma_n=0$ and $G\equiv 1$.
In other words, the theorem reproves the result that all degree 2
proper rational maps of balls are spherically equivalent to a polynomial map
taking the origin to the origin.

If degree $d=3$, then the denominator is normalized to
\begin{equation}
G(z) = 1 + \sum_{k=1}^n \sigma_k z_k^2.
\end{equation}
In this case the denominator is completely in normal form:
If $F=\frac{P}{G}$ and $\Phi=\frac{\Pi}{\Gamma}$
are spherically equivalent and in the form above, then $G=\Gamma$
and $\Pi = U \circ P \circ V$, where $U$ and $V$ are unitary matrices
and $G \circ V = G$.

As an example, let us study the case $n=N=1$ and $d=3$.  That is,
we find the normal form for proper rational maps
$f \colon \bB_1 \to \bB_1$ of degree 3.
A proper map of discs is a finite Blaschke product,
in this case a product of 3 factors.
As we require $f(0) =0$, one of these factors is $z$.  We also
require the denominator to be $1+ \sigma z^2$ ($\sigma \geq 0$ and $\sigma < 1$
so that the denominator does not vanish on the ball).
Thus the other two factors are the Blachke factors
taking zero to $\pm i\sqrt{\sigma} \in \bB_1$.
So the normal form is
\begin{equation} \label{eq:deg3n1normform}
f(z) =
z \left(
\frac{i\sqrt{\sigma}-z}{1-i\sqrt{\sigma}\, z}
\right)
\left(
\frac{-i\sqrt{\sigma}-z}{1+i\sqrt{\sigma}\, z}
\right)
=
\frac{\sigma z + z^3}{1+ \sigma z^2} ,
\qquad
0 \leq \sigma < 1 .
\end{equation}
If $\sigma = 0$, the map is simply $z^3$.
The map is in normal form as we said above and switching the sign of $z$
(the only unitary allowed on the source if $\sigma > 0$)
can be counteracted by a unitary on the target disc.
Thus \eqref{eq:deg3n1normform} is a complete normal form of third degree proper maps
of the disc up to spherical equivalence. 
The corresponding exhaustion function $\Lambda$ is
\begin{equation}
\Lambda(z,\bar{z})
=
\frac{
\sabs{1+\sigma z^2}^2-\sabs{\sigma z + z^3}^2
}{
(1-\sabs{z}^2)^3
} .
\end{equation}
This $\Lambda$ has a unique critical point (a minimum) at the
origin.
Moreover, the only degree 3 polynomial proper map of $\bB_1$ to $\bB_1$
is $e^{i\theta} z^3$.  Hence, the only map \eqref{eq:deg3n1normform}
spherically equivalent to a polynomial is the one where $\sigma = 0$.

As we mentioned above, D'Angelo observed that
any polynomial not zero on the closed ball
is the denominator of a rational proper map written in lowest terms.
By tensoring with the identity
we ensure that the map takes the origin to the origin.
That is, for any
denominator in normal form $G(z) = 1 + \sum_{k=1}^n \sigma_k z_k^2 + E(z)$ that is not
zero on the closed ball, a numerator exists giving a proper rational map of
balls (in lowest terms) taking the origin to the origin, hence one in normal form.
The degree of the numerator depends on the $\sigma_k$ and $E$.
D'Angelo (see Chapter 2 section 14 of \cite{DAngelo:spheresbook})
has worked out the following rather interesting example\footnote{D'Angelo
uses $1-\lambda z_1 z_2$, which we put into our normal form.}:
Given $\sigma \geq 0$, a polynomial in normal form
\begin{equation}
G(z) = 1 + \sigma z_1^2 + \sigma z_2^2
\end{equation}
is nonzero on the closed
unit ball whenever $\sigma < 1$, and consequently 
it is the denominator of a proper rational map
$\bB_2$ to $\bB_N$.  However, the degree $d$ of this map necessarily goes to
infinity as $\sigma$ approaches $1$.  A third degree map exists if $\sigma <
\frac{\sqrt{3}}{2}$.
A natural question given our normal form is, therefore, which possible denominators of
degree 2 in normal form,
$G(z) = 1 + \sum_{k=1}^n \sigma_k z_k^2$,
are denominators of third degree maps taking the origin to the origin.
D'Angelo~\cite{DAngelo:preprint}
recently made a systematic study of how the denominator is
determined by the numerator, and in particular gives a method
for constructing the precise equations needed, working out
the degree 3 case in detail.
For our purposes, we content ourselves with proving in section~\ref{section:exist},
that a degree 3 numerator taking the origin to the origin
corresponding to the denominator (in lowest terms)
exists for all small enough $\sigma_k$.

Finally, we remark that it is common to switch back and forth
between rational proper maps of balls $\bB_n \to \bB_N$ and rational maps
of spheres $S^{2n-1} \to S^{2N-1}$.  When $n > 1$, the only difference
between these two setups is the existence of a constant sphere map not
corresponding to a proper map of balls.  If $n=1$, then there are
many sphere maps that are not proper maps of balls.  For example,
$\frac{1}{z}$ takes the circle to the circle, but not the unit disc to the
unit disc.  The results of this paper hold for any $n$ as long as they are
stated for proper maps of balls.  They hold for rational proper maps of spheres
if $n > 1$.

The author would like to thank John D'Angelo, Dusty Grundmeier,
and Han Peters for many discussions on this and related topics, which led to 
the present result, Abdullah Al Helal for reading the manuscript and
asking insightful questions, and the
referee for careful reading of the proofs and suggesting where further
detail could be given.


\section{Preliminaries on proper maps of balls}

When 
$\frac{p}{g} \colon \bB_n \to \bB_N$ is a proper rational map written
in lowest terms, we consider
the real polynomial $\sabs{g(z)}^2-\snorm{p(z)}^2$.
We will call this real polynomial the \emph{underlying form} of
$\frac{p}{g}$.
Note that normally
$\snorm{p(z)}^2-\sabs{g(z)}^2$ is considered in the literature,
but its negative will be more convenient for us here.
As $f$ maps to the ball, the underlying form is positive on $\bB_n$,
and as $f$ is proper it is zero on the sphere.  That is,
\begin{equation}
\sabs{g(z)}^2- \snorm{p(z)}^2 = 0 \quad \text{on} \quad \snorm{z}^2 = 1 ,
\end{equation}
or in other words, there exists a polynomial $q(z,\bar{z})$ such that
\begin{equation}
\sabs{g(z)}^2- \snorm{p(z)}^2 = q(z,\bar{z}) \bigl( 1-\snorm{z}^2 \bigr) .
\end{equation}

The techniques in this paper are based on the following well-known idea.
A real polynomial $r$ can be written as
\begin{equation}
r(z,\bar{z})=
\sum_{\alpha,\beta} c_{\alpha,\beta} z^\alpha \bar{z}^\beta
=\snorm{h(z)}^2
\end{equation}
for a holomorphic polynomial
$h \colon \C^n \to \C^N$ if and only if the matrix of coefficients
$[c_{\alpha,\beta}]$ is
positive semidefinite, and the
rank of $[c_{\alpha,\beta}]$ is the
least $N$ necessary.
More generally, if $[c_{\alpha,\beta}]$ has $a$ positive and $b$ negative
eigenvalues, then $r$ can be written as a difference
$\snorm{h(z)}^2 - \snorm{u(z)}^2$, where $h$ has $a$ components and $u$ has
$b$ components.  This fundamental decomposition follows from elementary linear algebra.
We write the matrix of coefficients as a sum of rank one matrices,
each of which is $\pm$ an outer product of a vector with itself.
See~\cites{HornJohnson}.

Suppose $\frac{p}{g} \colon \bB_n \to \bB_N$ is a proper rational map, and
suppose that the components of $p$ together with $g$ are linearly independent, that is,
the image of $\frac{p}{g}$ is not contained in a complex hyperplane.
If it were, a linear fractional automorphism of the ball can move this
hyperplane to the hyperplane $\{ z_N = 0 \}$.  To make a long story short,
after composing with an automorphism, we may remove any
zero components and we can assume that $(p_1,\ldots,p_N,g)$ are linearly
independent.
This minimal $N$ is called the \emph{embedding dimension} for
the map.
We also remark that fixing $d$ we find that the maximal possible embedding
dimension for a degree $d$ map is one less (to account for $g$) than
the dimension of the space of polynomials of degree $d$.
If $N$ is the
embedding dimension of the rational proper map $\frac{p}{g}$,
then the matrix of coefficients of
the underlying form
$\sabs{g(z)}^2 - \snorm{p(z)}^2$ has
$1$ positive and $N$ negative eigenvalues.

The underlying form 
$\sabs{g(z)}^2 - \snorm{p(z)}^2$ can be rescaled by a real positive constant
without changing the map $\frac{p}{g}$ as such rescaling can be done by
rescaling $p$ and $g$ by that same constant.
The value $\sabs{g(0)}^2-\snorm{p(0)}^2$ is some positive quantity,
therefore, by such a rescaling we may assume
$\sabs{g(0)}^2-\snorm{p(0)}^2 = 1$.
Once we normalize the polynomial in this way, we ask what does it mean
if two proper maps have the same underlying form.
In \cite{Lebl:normal}, the author proved that the underlying forms are equal
(with the value at zero normalized as above) if and only if 
the maps themselves differ by a target automorphism.

\begin{lemma} \label{lemma:targetautform}
Suppose
$\frac{p}{g} \colon \bB_n \to \bB_N$ 
and $\frac{P}{G} \colon \bB_n \to \bB_N$ are proper rational maps
written in lowest terms such that
$\sabs{g(0)}^2-\snorm{p(0)}^2 = 1$ and
$\sabs{G(0)}^2-\snorm{P(0)}^2 = 1$.
Then there exists a $\tau \in \Aut(\bB_N)$ such that
\begin{equation}
\tau \circ \frac{p}{g} = \frac{P}{G}
\qquad
\text{if and only if}
\qquad
\sabs{g(z)}^2-\snorm{p(z)}^2 = 
\sabs{G(z)}^2-\snorm{P(z)}^2 .
\end{equation}
\end{lemma}

The lemma holds in more generality, the target can be an arbitrary
hyperquadric, see \cite{Lebl:normal} or \cite{DX}.
The power of this lemma is that it reduces the problem of
classification of proper maps up to pre and post composition by automorphisms
to classification of the underlying forms up to precomposition with an
automorphism.  We have eliminated the target automorphism group from the
problem.

Since $\Aut(\bB_N)$ is transitive, for
any proper ball map $\frac{p}{g}$, there exists
a $\tau \in \bB_N$ such that 
$\tau \circ \frac{p}{g} = \frac{P}{G}$, where $\frac{P}{G}$
is a proper ball map with $P(0) = 0$, $G(0) = 1$.
By Lemma~\ref{lemma:targetautform},
$\abs{g(z)}^2-\snorm{p(z)}^2 = 
\abs{G(z)}^2-\snorm{P(z)}^2$ (as long as
$\abs{g(0)}^2-\snorm{p(0)}^2 = 1$).
Instead of looking for $\tau$, we simply read off the $G$
by looking at the coefficients of the
polynomial $r(z,\bar{z}) = \sabs{g(z)}^2-\snorm{p(z)}^2$.

\begin{lemma} \label{lemma:Gfrommatrix}
Suppose
$\frac{p}{g} \colon \bB_n \to \bB_N$ 
is a proper rational map of degree $d$
written in lowest terms such that $p(0) = 0$ and $g(0)=1$.
Let $r(z,\bar{z}) = \sabs{g(z)}^2-\snorm{p(z)}^2$
and
$r(z,\bar{z}) = q(z,\bar{z}) \bigl(1-\snorm{z}^2\bigr)$.
Then $g$ is of degree $d-1$ or less, and
$g(z) = r(z,0) = q(z,0)$.
\end{lemma}

In particular, the coefficients of the denominator $g$ for the $p$ that takes the
origin to the origin are given by the row or column of the
coefficient matrix of $r$ or $q$ corresponding to pure holomorphic or
antiholomorphic terms.
So we can read off~$g$ from the coefficient matrix of~$r$ or~$q$.
Another way to think about it is
that the pure terms of $r$ equal
$g(z) + \overline{g(z)} -1$.

\begin{proof}
That $g$ is of degree at most $d-1$ is the result of D'Angelo we mentioned
above.
That $g(z) = r(z,0) = q(z,0)$ follows by complexifying the
$\sabs{g(z)}^2-\snorm{p(z)}^2$ and plugging in $\bar{z}=0$ using the fact
that $p(0)=0$ and $g(0)=1$.
\end{proof}

Every automorphism of the unit ball $\bB_n$ is written as $U
\varphi_\alpha$,
where $U$ is a unitary matrix and
\begin{equation} \label{eq:varphia}
\varphi_{\alpha}(z) =
\frac{{\alpha}-L_{\alpha} z}{1-\langle z,{\alpha}\rangle},
\qquad
L_{\alpha}z = \left(1-\sqrt{1-\snorm{{\alpha}}^2}\right)
\frac{\langle z,{\alpha}\rangle}{\snorm{{\alpha}}^2} {\alpha} 
+
\sqrt{1-\snorm{{\alpha}}^2} z,
\end{equation}
where if ${\alpha}=0$, then $L_0=I$.
Note that $L_\alpha$ is a linear map.
The automorphism $\varphi_\alpha$ is the one where
$\varphi_\alpha(0) = \alpha$ and $\varphi_\alpha(\alpha) =  0$.  In fact,
$\varphi_\alpha$ is an involution.
Note that the coefficients of the numerator and denominator of the
map $\varphi_\alpha$ are continuous in ${\alpha} \in \bB^n$.

%
%
%

We need a version of a result of
Cima--Suffridge~\cite{CimaSuffridge} (or Chiappari~\cite{Chiappari}) that the
denominator does not vanish, although they consider mainly $n \geq 2$.
The one dimensional case is rather simple, so for completeness, we give a proof.
See also~\cite{DHX}.

\begin{lemma} \label{lemma:chiappari}
Suppose
$f = \frac{p}{g} \colon \bB_n \to \bB_N$ 
is a proper rational map written in lowest terms.
Then $g$ never vanishes on $\partial \bB_n$.
\end{lemma}

\begin{proof}
When $n > 1$, it is the result of
Cima--Suffridge~\cite{CimaSuffridge}, so suppose $n=1$.
If $g(z_0)=0$ for some $z_0 \in \partial \bB_1 = S^1$,
then, as $f$ is in lowest terms, at least one component 
of $f$ has a pole at~$z_0$.  But
$\snorm{f(z)} < 1$ for all $z \in \bB_1 = \D$.  So for any
component $f_j(z)$ we also have $\sabs{f_j(z)} < 1$, so $f_j$
cannot have a pole on the closure of $\D$.
\end{proof}

The lemma has the following stronger corollary about the quotient
polynomial, which is the result that we will require.

\begin{lemma} \label{lemma:qpositive}
Suppose
$\frac{p}{g} \colon \bB_n \to \bB_N$ 
is a proper rational map written in lowest terms.
Let $q$ be the quotient polynomial $q(z,\bar{z}) =
\frac{\sabs{g(z)}^2-\snorm{p(z)}^2}{1-\snorm{z}^2}$ as above.
Then $q > 0$ on the sphere $S^{2n-1}$,
and hence on $\overline{\bB_n}$.
\end{lemma}

It is not difficult to see that $q > 0$ on the ball $\bB_n$:
After all, $\frac{p}{g}$ takes $\bB_n$ to $\bB_n$ and so
$\sabs{g(z)}^2-\snorm{p(z)}^2 > 0$ and $1-\snorm{z}^2 > 0$ 
for $z \in \bB_n$.
The point of the lemma is that $q$ does not vanish on the boundary.

\begin{proof}
Write $f = \frac{p}{g}$.  By Lemma~\ref{lemma:chiappari},
the function $f$ is holomorphic on a neighborhood of $\overline{\bB_n}$.
In particular, the denominator is nonzero in a neighborhood of
$\overline{\bB_n}$.  Thus we write
\begin{equation}
1-\snorm{f(z)}^2 = \frac{q(z,\bar{z})}{\sabs{g(z)}^2}
\bigl( 1-\snorm{z}^2 \bigr) .
\end{equation}
Taking the radial derivative $\frac{\partial}{\partial r}$ for a point $z_0$
on the sphere we find
\begin{equation}
\left.\frac{\partial}{\partial r}\right|_{z=z_0}\Bigl[\snorm{f(z)}^2\Bigr] =
2 \frac{q(z_0,\bar{z}_0)}{\sabs{g(z_0)}^2} .
\end{equation}
The function
$\snorm{f(z)}^2$ is plurisubharmonic, and hence by the Hopf lemma,
its radial derivative on the sphere must be positive, and hence
$q(z_0,\bar{z}_0) > 0$.
\end{proof}


\section{The $\Lambda$-function}

For a proper ball map $f = \frac{p}{g}$ written in lowest terms,
we define the \emph{$\Lambda$-function corresponding to $f$} as before:
\begin{equation} \label{eq:Lambda2}
\Lambda(z,\bar{z}) = \Lambda_f(z,\bar{z}) =
\frac{\sabs{g(z)}^2-\snorm{p(z)}^2}{(1-\snorm{z}^2)^d} .
\end{equation}
The $\Lambda$ does not change at all when postcomposing $f$ with an automorphism.

\begin{lemma} \label{lemma:Lambda1}
Suppose $f \colon \bB_n \to \bB_N$ is a rational proper map and
$\tau \in \Aut(\bB_N)$.
Then $\Lambda_f = \Lambda_{\tau \circ f}$.
\end{lemma}

\begin{proof}
The proof follows from Lemma~\ref{lemma:targetautform},
as the underlying form $r$ is the same for both $f$ and $\tau \circ f$.
\end{proof}

For precomposing with an automorphism, the $\Lambda$, up to a constant,
transforms by simply precomposing with the same automorhism.
The point is that when precomposing both $r$ and $(1-\snorm{z}^2)^d$, when
we try to clear denominators we multiply by the same function, that is,
for $\Lambda$ there is no need to clear the denominators.

\begin{lemma} \label{lemma:Lambda2}
Suppose $f \colon \bB_n \to \bB_N$ is a rational proper map and
$\psi \in \Aut(\bB_n)$.
Then $\Lambda_f \circ \psi = C \Lambda_{f \circ \psi}$.
\end{lemma}

\begin{proof}
Suppose that $r(z,\bar{z})=\sabs{g(z)}^2-\snorm{p(z)}^2$ is the underlying form
of $f$, and $R(z,\bar{z}) = \sabs{G(z)}^2-\snorm{P(z)}^2$ is the underlying form
of $f \circ \psi$.

First suppose $\psi$ is a unitary matrix.  Then
$G(z) = g(\psi(z))$ and 
$P(z) = p(\psi(z))$, as precomposing with a unitary does not introduce any
denominators and does not change the value at 0.
Furthermore $(1-\snorm{\psi(z)}^2)^d = (1-\snorm{z}^2)^d$.
Thus,
in this case $\Lambda_f \circ \psi = \Lambda_{f\circ \psi}$.

Now suppose that $\psi = \varphi_\alpha$ from \eqref{eq:varphia} for some $\alpha \in \bB_n$.
Let us first figure out how the underlying form $r$ transforms if we compose with $\varphi_{\alpha}$.
We compose $r$ and $\varphi_\alpha$ and then clear the denominators by multiplying by
$\sabs{1-\langle z,\alpha \rangle}^{2d}$.  To keep $r(0,0) = 1$ we also
multiply by a constant,
$\frac{1}{r(\alpha,\overline{\alpha})}$.
That is,
$r$ transforms to
\begin{equation}
R(z,\bar{z}) =
\frac{1}{r(\alpha,\overline{\alpha})} \sabs{1-\langle z,\alpha \rangle}^{2d}
r\bigl(
\varphi_{\alpha}(z),
\bar{\varphi}_{\alpha}(\bar{z})
\bigr) .
\end{equation}
Similarly,
$1-\snorm{z}^2$ transforms to
\begin{equation}
\frac{1}{1-\snorm{\alpha}^2} \sabs{1-\langle z,\alpha \rangle}^{2}
\left(
1-\norm{
\varphi_{\alpha}(z)
}^2
\right)
=
1-\snorm{z}^2 .
\end{equation}
That is, after clearing denominators,
$1-\snorm{z}^2$ is untouched by composing with an automorphism up
to a constant.

Therefore,
\begin{equation}
\Lambda \circ \varphi_{\alpha}
=
\frac{
r\bigl(
\varphi_{\alpha}(z),
\bar{\varphi}_{\alpha}(\bar{z})
\bigr)
}{
(1-\snorm{\varphi_\alpha(z)}^2)^d
}
=
\left(
\frac{r(\alpha,\overline{\alpha})}{ (1-\snorm{\alpha}^2)^d}
\right)
\,
\frac{
\sabs{G(z)}^2-\snorm{P(z)}^2
}{
(1-\snorm{z}^2)^d
} .
\end{equation}
\end{proof}

To show that $\Lambda$ is strongly rather than just strictly plurisubharmonic,
we need the following somewhat more general result that says that functions
like $\Lambda$ are, in fact, strongly convex near the boundary.

\begin{prop} \label{prop:convexnearbndry}
Suppose $B \subset \R^n$ is the unit ball and $h \colon \overline{B} \to \R$
is a $C^2$ function that is strictly positive on $\overline{B}$.
Define $w \colon B \to \R$ by $w(x) =
\frac{h}{(1-\snorm{x}^2)^k}$ for some $k \in \N$.  Then the least
eigenvalue of the Hessian of $w$ goes to $+\infty$ as $x$ goes to
$\partial B$.
\end{prop}

\begin{proof}
Write $H(w)$ for the Hessian.  We start with $n=2$.
Let $(x,y) \in \R^2$ be our coordinates.
By symmetry, we only need to do the computation at points where
$y=0$.
Write $\beta(x,y) = \frac{1}{(1-x^2-y^2)^k}$.
In the following, write the gradient as row vectors.
We compute (at a point where $y=0$),
\begin{equation}
\begin{split}
H(w) & = H(h\beta)
= \beta H(h) + (\nabla \beta)^t \nabla h + (\nabla h)^t \nabla \beta + h H(\beta) \\
& = 
\frac{1}{(1-x^2)^k}
H(h)
+
\begin{bmatrix}
\frac{h_x (2x) k}{(1-x^2)^{k+1}} & \frac{h_y (2x) k}{(1-x^2)^{k+1}} \\
0 & 0
\end{bmatrix}
+
\begin{bmatrix}
\frac{h_x (2x) k}{(1-x^2)^{k+1}} & 0 \\
\frac{h_y (2x) k}{(1-x^2)^{k+1}} & 0
\end{bmatrix}
\\
& \qquad
+
h
\begin{bmatrix}
\frac{2k(1-x^2)+4x^2k(k+1)}{(1-x^2)^{k+2}} & 0 \\
0 & \frac{2k}{(1-x^2)^{k+1}}
\end{bmatrix} .
\end{split}
\end{equation}
In other words, as $x$ approaches 1, the matrix $H(w)$ is
$\left[ \begin{smallmatrix} a & b \\ b & c \end{smallmatrix} \right]$
where $a$ and $c$ are positive, $b$ and $c$ grow as $\frac{1}{(1-x)^{k+1}}$
and $a$ grows as $\frac{1}{(1-x)^{k+2}}$.
Thus for $x$ close to $1$, we can
make the least eigenvalue arbitrarily large.
As the numerators in the expression for $a$, $b$, and $c$ are continuous functions on the
closed unit ball, this argument can be done uniformly as the point goes to
the boundary.

When $n > 2$, note that the least eigenvalue of $H(w)$ at $p$ is obtained by
minimizing $v^t H(w)|_p v$ over all unit vectors $v$.  Since every vector
$v$ based at a point $p$ lies in a $2$-dimensional subspace, and the set of
two dimensional subspaces is a compact set, the argument gives the result
for $n > 2$ as well.
\end{proof}

When the degree of the map is $d=1$, then $f = \tau \circ (z \oplus 0)$ for
some $\tau \in \Aut(\bB_N)$, so $r(z,\bar{z}) = 1-\snorm{z}^2$ and
$\Lambda(z,\bar{z}) \equiv 1$.  In particular, $\Lambda$ is not an exhaustion
function, nor is it strictly plurisubharmonic.  Therefore, $d > 1$
is required for the next two results.

\begin{lemma} \label{lemma:Lambda3}
Suppose $f \colon \bB_n \to \bB_N$ is a rational proper map
of degree $d > 1$.  Then
$\Lambda = \Lambda_f \colon \bB_n \to \R$ is a strongly plurisubharmonic
function such that
$\Lambda(z)$ goes to $+\infty$ as $z \to \partial \bB_n$.
In fact, $\Lambda$ is strongly convex near $\partial \bB_n$.
\end{lemma}

\begin{proof}
We start with strict plurisubharmonicity.
By the transitivity of $\Aut(\bB_n)$,
as composing $\Lambda$ with $\varphi_{\alpha}$ preserves strict plurisubharmonicity, and it
also transforms to another $\Lambda$-function for another proper map by
Lemma~\ref{lemma:Lambda2}, we
simply need to prove strict pluriharmonicity at the origin.
By Lemma~\ref{lemma:Lambda1} we can assume that $f = \frac{p}{g}$ where $p(0)=0$
and $g(0) = 1$.
A function is strictly plurisubharmonic if its restriction to every complex line
is strictly subharmonic.
We restrict to a
complex line through the origin, and so we can, without loss of generality,
assume $n=1$.
We follow the argument of the typical proof of Schwarz's lemma:
First, $f(z) = z F(z)$, as each component is divisible by $z$.
As $F$ is not constant ($d > 1$),
then $F(z)$ still takes the ball to the ball by the maximum principle.
In particular, $\snorm{F(z)} < 1$ for all $z \in \bB_1$.  Then
$\snorm{f'(0)} = \snorm{F(0)} < 1$.
Using the fact that $p(0)=0$ and $g(0)=1$ and the quotient rule we have
$f'(0) = p'(0)$, so $\snorm{p'(0)} < 1$.

Now compute the Laplacian and notice it is strictly positive
(using $d > 1$):
\begin{multline}
\frac{\partial^2}{\partial z \partial \bar{z}}\Big|_{z=\bar{z}=0} \Lambda(z,\bar{z})
=
d \, \sabs{g(0)}^2 - d \, \snorm{p(0)}^2 + \sabs{g'(0)}^2-\snorm{p'(0)}^2
\\
=
d  + \sabs{g'(0)}^2-\snorm{p'(0)}^2
>
d  + \sabs{g'(0)}^2-1 > 0 .
\end{multline}

Therefore, $\Lambda$ is strictly plurisubharmonic.
We need to show that it
is an exhaustion function.
Consider the quotient polynomial $q(z,\bar{z}) =
\frac{r(z,\bar{z})}{1-\snorm{z}^2}$.  Lemma~\ref{lemma:qpositive} says that
$q > 0$ on the closed ball $\overline{\bB_n}$, so
\begin{equation}
\Lambda(z,\bar{z}) = 
\frac{\sabs{g(z)}^2-\snorm{p(z)}^2}{(1-\snorm{z}^2)} \,
\frac{1}{(1-\snorm{z}^2)^{d-1}} =
q(z,\bar{z})
\frac{1}{(1-\snorm{z}^2)^{d-1}} .
\end{equation}
The result follows,
again using that $d > 1$.

Since $\Lambda(z,\bar{z}) = \frac{q(z,\bar{z})}{(1-\snorm{z}^2)^{d-1}}$,
$d-1 > 0$, and $q$ is positive on the closed unit
ball, Proposition~\ref{prop:convexnearbndry} implies that the least
eigenvalue of the (real) Hessian of $\Lambda$ goes to $+\infty$ as $z$
approaches the boundary.  Hence the least eigenvalue of the complex Hessian
also goes to $+\infty$, and $\Lambda$, being strictly plurisubharmonic on
all of $\bB_n$, must in fact be strongly plurisubharmonic.
\end{proof}

\begin{prop} \label{prop:g}
Suppose $g(z) = 1 + g_2 z^2 + g_3 z^3 + \cdots + g_k z^k$ is a polynomial
in one variable with no zeros on $\overline{\D}$.  Then $\sabs{g_2} < \frac{k}{2}$.
\end{prop}

\begin{proof}
Write $g(z) = (1-r_1 z)(1-r_2 z) \cdots (1-r_k z)$ where $\sabs{r_j} < 1$
for all $j$, and $r_1 + r_2 + \cdots + r_k = 0$.
Write
\begin{equation}
0 =
\left( \sum_{j=1}^k r_j \right)^2 = 
\sum_{j=1}^k r_j^2
+2
\left(
\sum_{1 \leq j < \ell \leq k} r_j r_\ell
\right) .
\end{equation}
Using this equality and
computing $g_2$ in terms of $r_j$s, we find
\begin{equation}
\sabs{g_2}
=
\abs{
\sum_{1 \leq j < \ell \leq k} r_j r_\ell 
}
=
\frac{1}{2}
\abs{
\sum_{j=1}^k r_j^2
}
\leq
\frac{1}{2}
\left(
\sum_{j=1}^k \sabs{r_j}^2
\right)
< \frac{k}{2} .
\end{equation}
\end{proof}

\begin{lemma} \label{lemma:Lambda4}
Suppose $f \colon \bB_n \to \bB_N$ is a rational proper map
of degree $d > 1$.  Then 
$\Lambda = \Lambda_f \colon \bB_n \to \R$ has
a unique critical point (a minimum) in $\bB_n$.
\end{lemma}

\begin{proof}
Since $\Lambda \to \infty$ as $z$ approaches $\partial \bB_n$,
there must be at least one critical point,
and so after composing with an automorphism we can assume it is at the
origin.  We restrict to an arbitrary complex line through the origin.
If we show that the determinant of the (real) Hessian of this restricted function
at the origin is strictly positive, then since the Laplacian (the trace of the Hessian)
is also strictly positive by Lemma~\ref{lemma:Lambda3},
the critical point is a strict local minimum.
Hence, without loss of generality
assume that $n=1$.

Again, by Lemma~\ref{lemma:Lambda1}, we can assume that $f = \frac{p}{g}$ where $p(0)=0$
and $g(0) = 1$.  Write
\begin{equation}
g(z) = 1 + g_1 z + g_2 z^2 + \cdots + g_{d-1} z^{d-1} .
\end{equation}
Using
\begin{equation}
\frac{1}{(1-\sabs{z}^2)^d} = 
1 + d z \bar{z}
+ \text{ (higher order terms),}
\end{equation}
and
\begin{equation}
\snorm{p(z)}^2 = \snorm{p'(0)}^2 z \bar{z}
+ \text{ (higher order terms),}
\end{equation}
we expand $\Lambda$ at the origin up to the second order:
\begin{multline}
\Lambda(z,\bar{z})
=
1+g_1 z + \bar{g}_1 \bar{z} + g_2 z^2 + \bar{g}_2 \bar{z}^2
+ \bigl(d + \sabs{g_1}^2 - \snorm{p'(0)}^2 \bigr) z \bar{z}
\\
+
\text{ (higher order terms).}
\end{multline}
The origin is a critical point of $\Lambda$, and so $g_1 = 0$.
The (real) Hessian determinant of $\Lambda$ at the origin is
$4 (d- \snorm{p'(0)}^2)^2
-
16 \sabs{g_2}^2$.
With $g_1=0$, the function $g$ satisfies the hypothesis of
Proposition~\ref{prop:g}.  Hence, $\sabs{g_2} < \frac{d-1}{2}$.
Furthermore, as in Lemma~\ref{lemma:Lambda3}, $\snorm{p'(0)} < 1$.
Therefore, the (real) Hessian determinant of $\Lambda$ at the origin
is positive:
\begin{equation}
4 {\bigl(d- \snorm{p'(0)}^2\bigr)}^2
-
16 \sabs{g_2}^2
>
4 (d- 1)^2
-
16 \sabs{g_2}^2 > 0.
\end{equation}

As $\Lambda$ has a strict local minimum at every critical point
and $\Lambda \to \infty$ as $z$ approaches the boundary,
the critical point must be unique.
This follows from what is now known as Courant's Mountain pass theorem,
see \cite[p.\ 223]{Courant},
which asserts that if a function goes to infinity at the
boundary and has two strict minima, then it must also have a saddle.
Another way to see that the minimum is unique is to use Morse theory:
The closed ball (apply to a closed ball of radius $1-\epsilon$
and note that $\nabla \Lambda$ points outside of this ball on the boundary)
has Euler characteristic 1,
and since the only critical points are minima, the Euler characteristic
equals the number of the critical points.
\end{proof}

The four items of Theorem~\ref{thm:Lambda} are simply
lemmas
\ref{lemma:Lambda1},
\ref{lemma:Lambda2},
\ref{lemma:Lambda3}, and
\ref{lemma:Lambda4}.


\section{Normalizing the denominator}

The key point in Theorem~\ref{thm:normdenom} is that we can zero out the
linear terms of the denominator while making sure the map takes the origin
to the origin.

\begin{lemma} \label{lemma:alphaiscrit}
Suppose $f \colon \bB_n \to \bB_N$ is a rational proper map
of degree $d > 1$.
Then for some
$\tau \in \Aut(\bB_N)$, the map
$\tau \circ f \circ \varphi_\alpha$
takes the origin to the origin, and its denominator (when written in lowest terms)
has no linear terms
if and only if $\alpha \in \bB_n$
is a critical point of the corresponding $\Lambda$-function.
\end{lemma}

The map $\tau$ is simply the automorphism that takes $f(\alpha)$ to $0$.

\begin{proof}
Let $r(z,\bar{z}) = \sabs{g(z)}^2-\snorm{p(z)}^2$.
Consider
\begin{equation}
R(z,\bar{z}) =
\sabs{1-\langle z,\alpha \rangle}^{2d}
r\bigl(
\varphi_{\alpha}(z),
\bar{\varphi}_{\alpha}(\bar{z})
\bigr) ,
\end{equation}
that is, consider the transformed form up to a constant that we forget for simplicity.
So up to a constant, by Lemma~\ref{lemma:Gfrommatrix}, we have that the
new denominator is (up to a constant) $R(z,0)$.  The new denominator has
zero linear coefficients if for all $j$
\begin{equation}
\frac{\partial R}{\partial z_j} \Big|_{z=\bar{z}=0} = 0 .
\end{equation}

We differentiate 
\begin{equation}
\begin{split}
\frac{\partial}{\partial z_j}
R(z,\bar{z}) & =
d(-\bar{\alpha}_j)\bigl(1-\langle z, \alpha \rangle\bigr)^{d-1}
\bigl(1-\langle \alpha, z \rangle\bigr)^{d}
r\bigl(
\varphi_{\alpha}(z),
\bar{\varphi}_{\alpha}(\bar{z})
\bigr)
\\
&
\phantom{={}}
+
\bigl(1-\langle z, {\alpha} \rangle\bigr)^{d}
\bigl(1-\langle {\alpha}, z \rangle\bigr)^{d}
\,
\nabla_z r\big|_{(
\varphi_{\alpha}(z),
\bar{\varphi}_{\alpha}(\bar{z})
)}
\cdot
\frac{\partial \varphi_{\alpha}}{\partial z_j}
.
\end{split}
\end{equation}
By $\nabla_z$ we mean the gradient in the $z$ variables
(and not the $\bar{z}$ variables).
To find the $z$ coefficients of $R$, we set $z=0$ and $\bar{z}=0$ above,
and then set the whole expression to zero:
\begin{equation} \label{eq:maineqforalpha}
d(-\bar{\alpha}_j)
r(\alpha,\bar{\alpha})
+
\nabla_z r \big|_{(\alpha,\bar{\alpha})}
\cdot
\frac{\partial \varphi_{\alpha}}{\partial z_j}\Big|_{z=0}
=0 .
\end{equation}
We now wish to show that
this solution is precisely the critical point of
\begin{equation}
\Lambda(z,\bar{z}) =
\frac{r(z,\bar{z})}{(1-\snorm{z}^2)^d} .
\end{equation}

For every unitary matrix $V$, we have $\varphi_{V\alpha}(Vz) = V \varphi_{\alpha}(z)$.
So if there is a solution $\alpha$ to \eqref{eq:maineqforalpha},
we rotate our map by $V$
and assume that $\alpha = (\alpha_1,0,\ldots,0)$.
We differentiate $(\varphi_\alpha)_k$, the $k$th component of
$\varphi_\alpha$.
\begin{equation}
\frac{\partial (\varphi_\alpha)_k}{\partial z_j}  \Big|_{z=0}
=
\begin{cases}
-\bigl(1-\sabs{\alpha_1}^2\bigr) & \text{if } k=j=1 , \\
\sqrt{1-\sabs{\alpha_1}^2} & \text{if } k=j \text{ and } k\not=1 , \\
0 & \text{else.}
\end{cases}
\end{equation}
For any $j$,
\begin{equation}
\frac{\partial \Lambda}{\partial z_j} =
\frac{(1-\snorm{z}^2)^d \frac{\partial r}{\partial z_j} - d(-\bar{z}_j)
(1-\snorm{z}^2)^{d-1} r}{(1-\snorm{z}^2)^{2d}}
.
\end{equation}
If $j= 1$, then \eqref{eq:maineqforalpha} becomes
\begin{equation}
d(-\bar{\alpha}_1)
r(\alpha,\bar{\alpha})
-
\bigl(1-\sabs{\alpha_1}^2 \bigr)
\frac{\partial r}{\partial z_1} \Big|_{(\alpha,\bar{\alpha})}
=0 .
\end{equation}
And that is equivalent to
$\frac{\partial \Lambda}{\partial z_1}\big|_{(\alpha,\bar{\alpha})} = 0$.

Similarly when $j > 1$, then
\eqref{eq:maineqforalpha} becomes
\begin{equation}
\sqrt{1-\sabs{\alpha_1}^2}
\,
\frac{\partial r}{\partial z_j} \Big|_{(\alpha,\bar{\alpha})}
=0 .
\end{equation}
And that is also equivalent to
$\frac{\partial r}{\partial z_j} \Big|_{(\alpha,\bar{\alpha})}=0$ and
hence
$\frac{\partial \Lambda}{\partial z_j}\big|_{(\alpha,\bar{\alpha})} = 0$
(noting that $\alpha_j = 0$).

Repeating the argument by differentiating with respect to $\bar{z}_j$,
we obtain that the conjugate of \eqref{eq:maineqforalpha} is
equivalent to $\frac{\partial \Lambda}{\partial
\bar{z}_j}\big|_{(\alpha,\bar{\alpha})} = 0$.
\end{proof}

We now prove Theorem~\ref{thm:normdenom}.

\begin{proof}[Proof of Theorem~\ref{thm:normdenom}]
If $d = 1$, the theorem is trivial, so assume the degree $d > 1$.
Write $\Lambda$ as before.  By Theorem~\ref{thm:Lambda}, 
$\Lambda$ has a unique critical point $\alpha$.
By Lemma~\ref{lemma:alphaiscrit}, the critical point gives
the unique automorphism $\varphi_\alpha$ such that if we
let $\tau$ be an automorphism that takes $f(\alpha)$ to $0$, and we write
$\tau \circ f \circ \varphi_\alpha = \frac{P}{G}$ in lowest terms, where
$G(0)=1$, then $P(0) = 0$ and $G$ has no linear terms.  That is,
\begin{equation}
G(z) = 1 + G_2(z) + G_3(z) + \cdots + G_{d-1}(z) ,
\end{equation}
where $G_j$ are homogeneous of degree $j$.
By the uniqueness of the critical point of $\Lambda$, the only
automorphisms that we can precompose with to keep $G$ in this form are
unitary matrices.
If we are only allowed to precompose with unitaries, the only
automorphisms on the target are also unitaries as they must fix the origin.

Consider $z$ as a column vector, then
\begin{equation}
G_2(z) = z^t C z
\end{equation}
for some complex symmetric matrix $C$ (the $z^t$ is the regular transpose).
Applying a unitary $V$, we find that
$G_2$ transforms to
\begin{equation}
z^t V^tCV z .
\end{equation}
It is an elementary result in linear algebra (see e.g.~\cite{HornJohnson}
Corollary 4.4.4 part c)
that a symmetric matrix can be
diagonalized by congruence with unitary matrices,
where the diagonal entries are all nonnegative and sorted according size.
The diagonal entries are the singular values (moduli of the eigenvalues) of
the matrix $C$.
In other words, after such a unitary
$G$ transforms to
\begin{equation}
G(z) = 1 + \sum_{k=1}^n \sigma_k z_k^2 + \text{ (higher order terms)},
\end{equation}
where $0 \leq \sigma_1 \leq \cdots \leq \sigma_n$.
By restricting to $z_1=\cdots=z_{n-1}=0$, we obtain a polynomial in 
one variable of degree $d-1$ or that satisfies the hypotheses of
Proposition~\ref{prop:g}.  Thus $\sigma_n \leq \frac{d-1}{2}$.
The only unitary matrices that can still be applied
that preserve the normal form are those that preserve this quadratic part.
\end{proof}


\section{Existence of third degree maps} \label{section:exist}

As we mentioned, any polynomial function $g(z)$ that is
non-zero on the closed ball $\overline{\bB_n}$ is the
denominator of a proper map of balls written in lowest terms, see
D'Angelo~\cite{DAngelo:spheresbook}.  The degree of the resulting map
depends not only on the degree of $g$, but also on the coefficients.
So a natural question is: Given
\begin{equation}
g(z) = 1 + \sum_{k=1}^n \sigma_k z_k^2 ,
\end{equation}
does there exist a degree 3 polynomial map $p$ such that $p(0)=0$ and $f = \frac{p}{g}$
is a rational proper map of balls written in lowest terms?
We show that for small enough
$\sigma_k$, such a map always exists.

For $\frac{p}{g}$ to be a proper map of balls, $g$ cannot be
zero on $\overline{\bB_n}$, and so Proposition~\ref{prop:g} gives a crude upper
bound for $\sigma_k$, that is, $\frac{d-1}{2}$.
For $d=3$, we find that $\sigma_k < 1$.
But as mentioned in the introduction,
for $g(z) = 1+\sigma z_1^2 + \sigma z_2^2$,
the degree of the numerator
required for a proper map to exist goes to infinity as $\sigma$ approaches 1.
For a degree 3 numerator to exist, the actual inequalities are
complicated except in the simplest case of $n=1$.

For $n=1$, the problem is easy to solve.
If $f \colon \bB_1 \to \bB_N$ is a rational proper map with
denominator $g(z) = 1 + \sigma z^2$, then $\sigma < 1$ as mentioned above.
To show existence it is enough to consider $N=1$.  In the introduction,
we computed the normal form for degree 3 proper rational maps
$f \colon \bB_1 \to \bB_1$, and we showed one exists with
denominator $1+\sigma z^2$ for every $0 \leq \sigma < 1$.

For $n > 1$,
the exact inequalities on $\sigma_k$ are more complicated.
From the proof below, these are inequalities guaranteeing a
certain matrix with some extra variables is positive semidefinite.
For any particular $n$ it is possible
to constructively write down the particular inequalities.
We dispense with attempting to
write down the exact inequalities and simply show that for
small enough $\sigma_k$, a degree 3 numerator $p$ always exists
making $\frac{p}{g}$ a proper map of balls written in lowest terms.
We refer the reader to a recent new preprint by
D'Angelo~\cite{DAngelo:preprint} for more precise information.

Recall that the $N$ required (after having fixed an $n$) is finite.
That is, the maximal $N$ is the number of
nonconstant degree 3 or lower monomials in $n$ variables.
That is, $N+1$ is the size of the matrix of coefficients 
of the underlying form $r(z,\bar{z})$ corresponding to degree~3 maps.

\begin{prop} \label{prop:exist}
\pagebreak[2]
Given $n$, there exists an $\epsilon > 0$ such that whenever
$0 \leq \sigma_1 \leq \cdots \leq \sigma_n < \epsilon$,
then there exists a degree 3 (or lower) polynomial $p \colon \C^n \to \C^N$
(where $N$ depends only on~$n$)
such that $p(0) = 0$ and
\begin{equation}
z \mapsto
\frac{p(z)}{1+\sum_{k=1}^n \sigma_k z_k^2}
\end{equation}
is a rational proper map of $\bB_n$ to $\bB_N$ written in lowest terms.

Moreover, the set of $\sigma_1,\ldots,\sigma_n$ for which 
a degree $p$ numerator exists as above is a semialgebraic set.
\end{prop}

We recall that a semialgebraic set is a subset of $\R^m$ given by a finite
set of polynomial equalities and inequalities.  A key result on semialgebraic sets is
the Tarski--Seidenberg theorem, which says that a projection of a
semialgebraic set onto a subspace is itself semialgebraic.

\begin{proof}[Proof of Proposition~\ref{prop:exist}]
Given a denominator $g(z) = 1+\sum_{k=1}^n \sigma_k z_k^2$,
we have a solution $p$ if $\norm{\frac{p(z)}{g(z)}}^2 = 1$ on the
sphere.  That is, if $\sabs{g(z)}^2-\snorm{p(z)}^2 = 0$ when
$\snorm{z}=1$.  In other words, $p$ is a solution when there is a
polynomial $q(z,\bar{z})$ such that
\begin{equation}
\sabs{g(z)}^2-\snorm{p(z)}^2 = q(z,\bar{z}) \bigl( 1-\snorm{z}^2 \bigr) .
\end{equation}
A degree 3 polynomial $p$ as in the
theorem exists if and only if there exists a $q$ of bidegree $(2,2)$ (degree
2 in $z$ and 2 in $\bar{z}$) such that
\begin{equation} \label{eq:tobepositive}
\sabs{g(z)}^2
-
q(z,\bar{z}) \bigl( 1-\snorm{z}^2 \bigr)
\end{equation}
has a positive semidefinite matrix of coefficients and therefore is a sum of
(hermitian) squares.  If the matrix also has no pure
(holomorphic or antiholomorphic) terms, the
resulting polynomial map $p$ will take the origin to the origin.

Let $[q_{\alpha,\beta}]$ be the matrix of coefficients of $q$.  Let
$[c_{\alpha,\beta}]$ be the matrix of coefficients of
\eqref{eq:tobepositive}.
The matrix
$[c_{\alpha,\beta}]$
being positive definite is a set of inequalities on the elements
$c_{\alpha,\beta}$,
which are polynomials in the various $q_{\alpha,\beta}$ and the $\sigma_k$,
and these inequalities give a semialgebraic set.
The set of possible $\sigma$s is the projection of this
semialgebraic set onto $(\sigma_1,\ldots,\sigma_n)$.  By
Tarski--Seidenberg, this projection is a semialgebraic set
and we have proved the ``Moreover.''

We move to the main conclusion of the proposition.  First, there
exists a third degree polynomial proper map of balls where the matrix
for $\snorm{p(z)}^2$ in the setup above is diagonal, that is,
where each component of the map $p$ is a single monomial.
Furthermore, we wish all the monomials be used, except the
constant.  One possibility for the form corresponding to $\snorm{p(z)}^2$ is
\begin{equation}
\frac{1}{3} \bigl(\snorm{z}^6 + \snorm{z}^4 + \snorm{z}^2 \bigr) .
\end{equation}
Except for the diagonal term corresponding to the constant, all
the diagonal terms of matrix of coefficients are of size at least $\frac{1}{3}$.
Note that
\begin{equation}
1-
\frac{1}{3} \bigl( \snorm{z}^6 + \snorm{z}^4 + \snorm{z}^2 \bigr)
=
q(z,\bar{z}) \bigl( 1-\snorm{z}^2 \bigr) .
\end{equation}
The trick now is to perturb this form (perturb the matrix) while keeping it
of the form $\sabs{g(z)}^2-\snorm{p(z)}^2$ (for a different $p$ of course)
and making sure that it is still divisible by $1-\snorm{z}^2$.

Consider
\begin{equation}
r(z,\bar{z}) =
1-\frac{1}{3} \bigl(\snorm{z}^6 + \snorm{z}^4 + \snorm{z}^2 \bigr)
+
\sum_{k=1}^n \sigma_k (z_k^2 +\bar{z}_k^2) \bigl( 1-\snorm{z}^2 \bigr) .
\end{equation}
Clearly $r = 0$ when $\snorm{z}=1$.  It remains to show that it can be
written as $\sabs{g(z)}^2-\snorm{p(z)}^2$ where $g$ is the given
denominator.

In the matrix of coefficients for $\sabs{g(z)}^2 =
g(z)\bar{g}(\bar{z})$, 
the column and row corresponding to the constant monomial is the same
as $r(z,\bar{z})$.  That is because if we plug in $\bar{z}=0$
\begin{equation}
r(z,0) = 1
+ \sum_{k=1}^n \sigma_k z_k^2 = g(z) = g(z) \bar{g}(0) .
\end{equation}
The matrix of coefficients of
$-r(z,\bar{z})+\sabs{g(z)}^2$ has zeros in the corresponding column and row.
We need to show that $-r(z,\bar{z})+\sabs{g(z)}^2$ can be written
as $\snorm{p(z)}^2$ for some $p \colon \C^n \to \C^N$, where $p(0)=0$.
The entries corresponding to mixed (neither holomorphic nor antiholomorphic) terms
in the matrix are of two types.
The diagonal entries come
from $\frac{1}{3} (\snorm{z}^6 + \snorm{z}^4 + \snorm{z}^2 )$ and are 
all of size~$\frac{1}{3}$ or greater, and they do not depend on $\sigma_k$.
The off-diagonal entries are all multiples of $\sigma_k$ for various $k$.
In particular, if $\sigma_k$ are all small enough, then 
the coefficient matrix for $-r(z,\bar{z})+\sabs{g(z)}^2$ is positive
semidefinite, the matrix has zeros in the pure
(holomorphic or antiholomorphic) terms, and
the $N \times N$ submatrix of mixed terms is positive
definite.
Therefore,
$-r(z,\bar{z})+\sabs{g(z)}^2$ can be written as a sum of $N$ hermitian
squares, that is, there exists a polynomial map $p$ such that
\begin{equation}
r(z,\bar{z}) = \sabs{g(z)}^2-\snorm{p(z)}^2 .
\end{equation}
Since the matrix we are decomposing to get $p$ is the $N \times N$
matrix of the mixed terms, all the components of $p$ vanish at the origin.
It remains to show that $\frac{p}{g}$ is in lowest terms.  If $g$ and all
components of $p$ have a common factor $h$, then dividing by $\sabs{h}^2$
would result in a form corresponding to a proper map of balls of
degree 2 or 1.  In such a case, the rank of the coefficient matrix of
$\frac{r}{\sabs{h}^2}$ is necessarily the same as the rank of $r$,
as the number of linearly independent squares in the expansion does
not change.  But for small
enough $\sigma_k$s, $r$ has a matrix that is of full rank (for bidegree
$(3,3)$).
If such an $h$ existed, the matrix
for $\frac{r}{\sabs{h}^2}$ is necessarily of strictly lower rank as
the bidegree is lower and all the coefficients for monomials of degree 3
are zero.  Thus no $h$ exists and $\frac{p}{g}$ is in lowest terms.
\end{proof}


\section{Rational and polynomial maps of maximal embedding dimension}
\label{section:poly}

The normal form does not necessarily answer the question of when is
a rational proper ball map equivalent to a polynomial.  In a certain
sense, most maps are equivalent to a polynomial map although that map
may not take the origin to the origin, and may involve the ball in different
coordinate patch of the projective space.
As we mentioned before, once the degree is 3 or higher,
there do exist maps that are not spherically equivalent to a polynomial.

Recall that for a fixed degree~$d$, there is a maximal embedding dimension for
a rational proper map of balls.
This dimension is exactly one less than the dimension of the space
of holomorphic polynomials of degree $d$,
as the embedding dimension for $\frac{p}{g}$ is
the dimension of the span of $(p_1,\ldots,p_N,g)$.
In a certain sense, the ``generic'' degree~$d$ map
has the maximal embedding dimension: It is given by a generic real polynomial
$r(z,\bar{z})$ of bidegree $(d,d)$ that is divisible by $1-\snorm{z}^2$
and has 1 negative eigenvalue, and a generic such form is of full rank.

We do not, however, get polynomial maps to the standard model of the ball.
Consider the following two other models of the unit ball,
that are both equivalent to $\bB_n$
via a linear fractional transformation.  First, the hyperquadric or
generalized ball with $1$ positive and $N-1$ negative eigenvalues:
\begin{equation}
\bB_{1,N-1} = \left\{
z \in \C^N : \sabs{z_1}^2 - \sabs{z_2}^2 - \cdots - \sabs{z_{N-1}}^2 < 1
\right\} .
\end{equation}
Second, the Heisenberg model of the ball
\begin{equation}
\bH_{N} = \left\{
z \in \C^N : \Re z_N > \sabs{z_1}^2 + \cdots + \sabs{z_{N-1}}^2 \right\} .
\end{equation}
The linear fractional maps that go between $\bB_n$, $\bB_{1,N-1}$, and
$\bH_N$ also transform the defining forms
$1-\snorm{z}^2$ for $\bB_N$,
$\sabs{z_N}^2 - 1 - \sabs{z_1}^2 - \cdots - \sabs{z_{N-1}}^2$ for
$\bB_{1,N-1}$,
and
$\Re z_N - \sabs{z_1}^2 - \cdots - \sabs{z_{N-1}}^2$ for $\bH_N$.
Therefore, 
Lemma~\ref{lemma:targetautform} holds with these forms instead of
$1-\snorm{z}^2$ if we replace the target ball by $\bB_{1,N-1}$ or $\bH_N$.

\begin{prop} \label{prop:poly}
\pagebreak[2]
Suppose $f \colon \bB_n \to \bB_N$ is a rational proper map of degree $d$
of maximal embedding dimension for degree $d$.
Then there exists a linear fractional map $\tau$ on $\C^N$ such that we get
one of the three following conclusions:
\begin{enumerate}[(i)]
\item
$\tau \circ f \colon \bB_n \to \bB_N$ is a polynomial proper map.
\item
$\tau \circ f \colon \bB_n \to \bB_{1,N-1}$ is a polynomial proper map.
\item
$\tau \circ f \colon \bB_n \to \bH_{N}$ is a polynomial proper map.
\end{enumerate}
\end{prop}

To decide which conclusion holds, write $f = \frac{p}{g}$, and
$r(z,\bar{z}) = \sabs{g(z)}^2-\snorm{p(z)}^2$ and normalize so that $r(0,0)
= 1$.  If the coefficient matrix of the mixed terms, that is,
$r(z,\bar{z}) - r(z,0) - r(0,\bar{z}) + 1$,
has rank less than $N$, then item (iii) holds.  Otherwise, we will see
below that $r(z,\bar{z})-\gamma$ is of rank $N$ for some $\gamma$.
Item (i) holds if $\gamma > 0$ and item (ii) holds if $\gamma < 0$.

\begin{proof}
Let $f = \frac{p}{g}$ as usual, normalize so that $g(0) = 1$.
By Lemma~\ref{lemma:targetautform}, to prove item (i) of the proposition,
we must show that
\begin{equation}
r(z,\bar{z})=\sabs{g(z)}^2 - \snorm{p(z)}^2 = \gamma \bigl(1 - \snorm{P(z)}^2\bigr)
\end{equation}
for some constant $\gamma > 0$ and some polynomial map $P(z)$ to $\C^N$.
To prove (ii), we need to show that
\begin{equation}
r(z,\bar{z}) = \gamma\bigl(\sabs{G(z)}^2 - 1 - \snorm{P(z)}^2\bigr)
\end{equation}
for some polynomial $G(z)$ and a polynomial map $P(z)$ to $\C^{N-1}$.
To prove (iii), we need to show that
\begin{equation}
r(z,\bar{z}) = \Re G(z) - \snorm{P(z)}^2
\end{equation}
for some polynomial $G(z)$ and
some polynomial map $P(z)$ to $\C^{N-1}$.
The proof lies in showing that one of the three possibilities holds.

Let $C$ be the hermitian coefficient matrix of the form $r(z,\bar{z})$.
That $N$ is the maximal embedding dimension for this degree,
means that $C$ has full rank.
The decomposition $\sabs{g(z)}^2 - \snorm{p(z)}^2$
is simply writing $C$ as a sum of hermitian rank one matrices.  That is,
if $\sZ = (z_1^d,\ldots,1)^t$ is the column vector of all monomials
of degree $d$ or less, then for a column vector $v$ composed of
the complex conjugates of the coefficients of $g$, we have $g(z) = v^* \sZ$
(where $v^*$ is the complex conjugate transpose), and
$\sabs{g(z)}^2 = \sZ^* v v^* \sZ$.
Hence, if $w_j$ is the vector of complex conjugates of the coefficients of
$p_j(z)$, then
\begin{equation}
C = v v^* - w_1 w_1^* - \cdots - w_N w_N^* .
\end{equation}
Note that $C$ is a full-rank matrix of size $(N+1) \times (N+1)$.
Let $J$ be the rank $1$ matrix that is the coefficient matrix of the constant
form $1$.  That is, $J$ is a matrix of all zeros except a~1 in the entry
corresponding to the constant.
Thus to find a decomposition giving (i) or (ii) we need to find 
a $\gamma$ such that $C - \gamma J$ is of rank $N$ (it can only be of rank
$N+1$ or $N$).

We are looking for zeros of $\det(C-\gamma J)$.  This function is constant
and has no zeros if and only if the matrix of coefficients of the mixed
terms $r(z,\bar{z}) - r(z,0) - r(0,\bar{z}) + 1$ is singular.  As
$r(z,0) + r(0,\bar{z}) - 1$ has rank 2, we find that in this case the matrix
of mixed terms must be of rank $N-1$ exactly.  But then, since 
$r(z,0) + r(0,\bar{z}) - 1$ has one positive and one negative eigenvalue,
the matrix of mixed terms must be negative semidefinite.
So write
\begin{equation}
r(z,\bar{z}) - r(z,0) - r(0,\bar{z}) + 1 = -\snorm{P(z)}^2 ,
\end{equation}
where $P$ is a polynomial map to $\C^{N-1}$.  Then write
\begin{equation}
r(z,0) + r(0,\bar{z}) - 1 = \Re G(z) ,
\end{equation}
where $G(z) = 2r(z,0) - 1$.  And so we obtain (iii).

So suppose that $\det(C-\gamma J)$ is not constant.
It is an affine linear function, so there must exist
a $\gamma$ where it vanishes, that is,
where $C-\gamma J$ is of rank $N$,
and $C-\gamma J$ is a sum of $N$ rank 1 matrices.
The signature of $C-\gamma J$ depends on the sign of $\gamma$.
If $\gamma < 0$, then $r(z,\bar{z}) + (-\gamma)$ must have one positive, $N-1$ negative
eigenvalues, so we could rewrite $r$ as
\begin{equation}
r(z,\bar{z}) = \snorm{G(z)}^2 -\sabs{\sqrt{-\gamma}}^2 -
\snorm{P(z)}^2
\end{equation}
for some polynomial $G(z)$ and a polynomial map $P(z)$ to $\C^{N-1}$.  That gives (ii).

Finally, suppose $\gamma > 0$.
Then $C-\gamma J$ must have $N$ negative eigenvalues.
By decomposing this matrix into rank 1 matrices we find the representation
\begin{equation}
r(z,\bar{z}) = \sabs{\sqrt{\gamma}}^2-\snorm{P(z)}^2
\end{equation}
for some polynomial $P$.  Part (i) follows.
\end{proof}

It is good to remark why the idea of the proof does not work if the
embedding dimension is not the maximum possible:
If $C$ is not of full rank,
then the rank of 
$C-\gamma J$ may never drop, but the rank of 
the matrix of mixed terms could still be of rank $N$ rather than $N-1$.

The set of proper maps equivalent to polynomial maps to $\bB_N$
(resp.\ $\bB_{1,N-1}$)
of maximal embedding dimension is open.
We say a rational proper map $f \colon \bB_n \to \bB_N$
is \emph{target equivalent to a polynomial proper map to $\bB_N$
(resp. $\bB_{1,N-1}$)} if there exists a linear fractional transformation
$\tau$ such that $\tau \circ f$ is a polynomial proper map to $\bB_N$ (resp.
$\bB_{1,N-1}$).
Below,
denote $\snorm{f}_{\bB_n} = \sup_{z \in \bB_n} \snorm{f(z)}$.

\begin{prop}
Suppose $f \colon \bB_n \to \bB_N$
is a rational proper map of degree $d$
of maximal embedding dimension for degree $d$
that is target equivalent to a polynomial proper map to $\bB_N$
(resp.\ $\bB_{1,N-1}$).  Then there exists an $\epsilon > 0$
such that if $F \colon \bB_n \to \bB_N$ is a rational proper map 
of degree $d$ where
$\snorm{f-F}_{\bB_n} < \epsilon$, then
$F$ is target equivalent to a polynomial proper map to $\bB_N$
(resp.\ $\bB_{1,N-1}$).
\end{prop}

\begin{proof}
Write $f = \frac{p}{g}$ and $F=\frac{P}{G}$ and write
the underlying forms
\begin{equation}
r(z,\bar{z})=\sabs{g(z)}^2 - \snorm{p(z)}^2
\qquad \text{and}\qquad
R(z,\bar{z})=\sabs{G(z)}^2 - \snorm{P(z)}^2 ,
\end{equation}
where we assume $r(0,0)=R(0,0)=1$ as usual,
which
also normalizes $\sabs{g(0)}$ and $\sabs{G(0)}$.
We further assume $g(0)$ and $G(0)$ to be positive.
The condition $\snorm{f-F}_{\bB_n} < \epsilon$
is equivalent to
the coefficients of the Taylor series at the origin up to any
fixed degree being within some~$\epsilon'$.
From that we find, as both $f$ and $F$ are of the same degree $d$,
that this condition is equivalent to 
the coefficients of the coefficient matrix of $r$ and $R$ being close.
The argument in the proof of Proposition~\ref{prop:poly}
means that we can choose either a positive or negative $\gamma$ such that
$\det(C-\gamma J)=0$.
A small enough perturbation of the underlying matrix also allows the choice of
such a $\gamma$ of the same sign, so the proposition follows.
\end{proof}

In particular, any rational proper map that is a small enough perturbation 
of a polynomial map of maximal embedding dimension for a given degree
is spherically equivalent to a polynomial proper map.




\end{document}